\documentclass[final]{siamltex}
\usepackage{amsmath}
\usepackage{mathrsfs}
\usepackage{amsfonts}
\usepackage[dvips]{graphicx}
\usepackage{psfrag}

\usepackage{fancyhdr}
\usepackage{amsfonts}
\usepackage{amsbsy}
\usepackage{amsgen}
\usepackage{amscd}
\usepackage{bbm}
\usepackage{euscript}

\usepackage{amsopn}
\usepackage{amstext}
\usepackage{amssymb}
\usepackage{amscd}
\usepackage{latexsym}

\newtheorem{example}[theorem]{Example}

\DeclareMathOperator{\Span}{span}
\DeclareMathOperator{\T}{T}
\DeclareMathOperator{\closure}{clo}
\DeclareMathOperator{\interior}{int}
\DeclareMathOperator{\cone}{cone}

\DeclareMathOperator{\conv}{conv}

\DeclareMathOperator{\tr}{tr}
\DeclareMathOperator{\rec}{rec}

\newcommand{\comment}[1]{}

\newcommand{\N}{\mathbb{N}}
\newcommand{\R}{\mathbb{R}}
\newcommand{\Proj}{\mathbb{S}}

\newcommand{\Sym}[1]{\ensuremath{\sym_{#1}}}
\newcommand{\Quotient}{\ensuremath{\mathfrak q}}
\newcommand{\sym}{\ensuremath{\mathscr S}}

\begin{document}

\title{The S-Procedure via Dual Cone Calculus}

\author{Raphael Hauser\footnotemark[1]}

\date{\today}

\renewcommand{\thefootnote}{\fnsymbol{footnote}}
\footnotetext[1]{Mathematical Institute, University of Oxford, 
24--29 St Giles', Oxford, OX1 3LB, United Kingdom, (hauser@maths.ox.ac.uk), and Pembroke 
College Oxford. The author was supported through grants GR/S34472 and EP/H02686X/1 
from the Engineering and Physical Sciences Research Council of the UK. }
\renewcommand{\thefootnote}{\arabic{footnote}}

\maketitle

\begin{abstract}
Given a quadratic function $h$ that satisfies a Slater condition, 
Yakubovich's S-Procedure (or S-Lemma) gives a characterization of all other quadratic functions 
that are copositive with $h$ in a form that is amenable to numerical computations. In this paper 
we present a deep-rooted connection between the S-Procedure and the dual cone calculus formula   
$(K_1\cap K_2)^*= K_1^*+K_2^*$, which holds for closed convex cones in $\R^2$. To establish the link 
with the S-Procedure, we generalize the dual cone calculus formula to a situation where $K_1$ is 
nonclosed, nonconvex and nonconic but exhibits sufficient mathematical resemblance to a closed 
convex cone. As a result, we obtain a new proof of the S-Lemma and an extension to Hilbert space 
kernels. 
\end{abstract}



%%%%%%%%%%%%%%%%%%%%%%
\section{Introduction} 
%%%%%%%%%%%%%%%%%%%%%%

Yakubovich's {\em S-Lemma} \cite{Yakubovich}, also called {\em S-Procedure}, is a well-known 
result from robust control theory that characterizes all quadratic functions that are copositive with 
a given other quadratic function. A function $g$ is called {\em copositive 
with $h$} if $h(x)\geq 0$ implies $g(x)\geq 0$. 

\begin{theorem}[S-Lemma, \cite{Yakubovich}]\label{S-Lemma}
Let $g,h:\R^n\rightarrow\R$ be quadratic functions such that 
$h(x_0)>0$ at some point $x_0\in\R^n$. 
Then $g$ is copositive with $h$ if and only if there exists $\xi\geq 0$ 
such that $g(x)-\xi h(x)\geq 0$ for all $x\in\R^n$.
\end{theorem}

Note that $g$ and $h$ are neither assumed to be convex nor homogeneous, and that the condition 
$g(x)-\xi h(x)\geq 0$ for all $x\in\R^n$ is easy to check, for a quadratic function 
$x\mapsto x^{\T}Qx+2\ell^{\T}x+c$ can always be formulated so that the matrix $Q$ is symmetric, 
and then the function is nonnegative everywhere on $\R^n$ if and only if the matrix 
$\bigl[\begin{smallmatrix}Q&\ell\\ \ell^{\T}&c\end{smallmatrix}\bigr]$ is positive semidefinite. 
The importance of this characterization is that it can be checked numerically. 

Theorem \ref{S-Lemma} arose as a generalization of earlier results by Finsler 
\cite{Finsler}, Hestenes \& McShane \cite{Hestenes} and Dines \cite{Dines}. Megretsky \& Treil 
\cite{Megretzky} later extended the result further. The S-Lemma has suprisingly powerful 
consequences in robust optimization and control theory, as this result allows to 
replace certain nonconvex optimization problems by convex polynomial 
time solvable ones, and semi-infinite programming problems by optimization models with 
finitely many constraints. Indeed, Theorem \ref{S-Lemma} says that in an optimization problem 
in which the coefficients $Q,\ell,c$ of the polynomial $g$ play the role of decision variables, 
the infinitely many constraints 
\begin{equation*}
g(x)\geq 0,\quad\forall x\in\R^n\text{ s.t. }h(x)\geq 0
\end{equation*}
can be replaced by a single matrix inequality 
\begin{equation*}
\begin{bmatrix}Q&\ell\\ \ell^{\T}&c\end{bmatrix}-\xi 
\begin{bmatrix}A&b\\ b^{\T}&d\end{bmatrix}\succeq 0, 
\end{equation*}
where $A,b,d$ are chosen such that $h(x)=x^{\T}Ax+2b^{\T}x+d$, and where $\xi$ is an 
auxiliary decision variable introduced by this lifting. 

For a overviews of the history of the S-Lemma and its applications, see \cite{Polik} and 
\cite{Derinkuyu}. Three existing known approaches to proving Theorem \ref{S-Lemma} 
described in \cite{Polik} are due to Yakubovich \cite{Yakubovich}, Ben-Tal \& Nemirovski 
\cite{Ben-Tal} and Sturm \& Zhang \cite{Sturm}, and Yuan \cite{Yuan}. 

In this paper we give a new proof of the S-Lemma that is based on a generalization of the dual 
cone calculus formula $(K_1\cap K_2)^*= K_1^*+K_2^*$, which is known to hold 
true for closed convex cones $K_1,K_2\subseteq\R^2$, to a situation where $K_1$ is nonclosed, 
nonconvex and nonconic but exhibits sufficient mathematical resemblance to a closed convex cone. 
For this purpose we introduce a weak notion of convexity, {\em homogenization-convexity}, the 
theory of which will  be developed in Section \ref{sec: homogenization-convexity2D}. Our 
proof extends quite straighforwardly to an S-Lemma for Hilbert space kernels. The techniques we 
employ are elementary. The main ideas of the proof merely require linear algebra in two dimensions. 
The S-Lemma and its extension to Hilbert space kernels are then obtained by a lifting. 

Among the existing proofs of the S-Lemma, Yakubovich's orginal proof is closest in spirit to the 
proof presented in this paper. Yakubovich employed a result of Dines \cite{Dines}, which shows that 
the joint range $\{(f(x), g(x)):\,x\in\R^n\}$ of two homogeneous quadratic functions $f,g$ 
on $\R^n$ is convex. Our own approach is based on showing that the projection of the set 
\begin{equation*}
\left\{\left[\begin{smallmatrix}x\\ 1\end{smallmatrix}\right]
\left[\begin{smallmatrix}x\\ 1\end{smallmatrix}\right]^{\T}:\,x\in\R^n\right\}
\end{equation*}
into a 2-dimensional subspace satisfies the weaker notion of homogenization-convexity. Once 
this is established, the S-Lemma follows from our generalized dual cone calculus formula. 

\subsection{Notation}\label{notation}

The inner product on any Hilbert space $V$ is denoted by 
$\langle\cdot,\cdot\rangle$. This inner product defines the canonical self-duality 
isomorphism on $V$ and the canonical norm $\|\cdot\|$. The topological closure and boundary 
of a set $C\subseteq V$ under the induced topology are denoted by $\closure[C]$ and 
$\partial C$. The convex, conic and homogeneous hulls of $C$ are denoted by 
\begin{align*}
\conv(C)&:=\left\{\sum_{i=1}^n\lambda_i x_i:\,n\in\N,\,
\lambda_i\geq 0,\, x_i\in C,\,\forall i,\,\sum_{i=1}^n\lambda_i=1\right\},\\
\cone(C)&:=\left\{\sum_{i=1}^n\lambda_i x_i:\,n\in\N,\,
\lambda_i\geq 0,\, x_i\in C,\,\forall i\right\},\\
\hom(C)&:=\left\{\tau x:\,\tau\geq 0,\,x\in C\right\}. 
\end{align*}
The relation between these three concepts is that $\cone(C)=\hom(\conv(C))$. 

\begin{definition} 
For any $C\subseteq V$ we refer to the set $\closure[\hom(C)]$ as the {\em homogenization} of $C$. 
\end{definition}

We denote the unit sphere in $(V,\langle\cdot,\cdot\rangle)$ by $\Proj(V)$ and the 
spherical projection by  
\begin{align*}
\Quotient:\;&V\setminus\{0\}\rightarrow\Proj(V)\\
x&\mapsto \frac{x}{\|x\|}.
\end{align*}
Note that the spherical projection is not defined at the origin of $V$. Nonetheless, by abuse 
of language, if $C\subseteq V$ we write $\Quotient(C)$ for $\Quotient(C\setminus\{0\})$. 
The set of {\em recession directions} of a set $C\subset V$ is given by 
\begin{equation*}
\rec(C):=\left\{s\in\Proj(V):\,\forall\tau,\varepsilon>0,\;\exists\,x\in C\text{ s.t. }\|s-\Quotient(x)\|<
\varepsilon,\,\|x\|>\tau\right\}. 
\end{equation*}
For any $x_1,x_2\in V\setminus\{0\}$ we write 
\begin{equation*}
[x_1,x_2]=\left\{\xi x_2+(1-\xi)x_1:\,\xi\in[0,1]\right\}
\end{equation*}
for the straight-line segment  between $x_1$ and $x_2$. For $y_1,y_2\in\Proj(V)$, we write 
$[y_1,y_2]:=\Quotient([x_1,x_2])$, where $x_1\in
\Quotient^{-1}(y_1)$ and $x_2\in\Quotient^{-1}(y_2)$. It is 
easy to check that the definition of $[y_1,y_2]$ does not depend on the 
specific choice of $x_1$ and $x_2$.  A subset $S\subset\Proj(V)$ is {\em spherically 
convex} if $[y_1,y_2]\subset S$ for all $y_1,y_2\in S$. 

%%%%%%%%%%%%%%%%%%%%%%%%%%%%%%%%%%%%%%%
\section{Homogenization-Convexity}\label{sec: homogenization-convexity2D}
%%%%%%%%%%%%%%%%%%%%%%%%%%%%%%%%%%%%%%%

Our approach to the S-Lemma hinges on a weak notion of convexity that we shall now define.

\begin{definition}\label{homogenization-convexity2D}
A set $C\subseteq\R^2$ is homogenization-convex if the homogenization $\closure[\hom(C)]$ of $C$ 
is a convex subset of $\R^2$. 
\end{definition}

A few alternative characterizations provide further insight:

\begin{lemma}\label{lem:trivial2D}
The following conditions on a set $C\subseteq\R^2$ are equivalent:
\begin{itemize}
\item[i) ] $C$ is homogenization-convex.
\item[ii) ] $\closure[\Quotient(C)]$ is spherically convex in $\Proj(V)$
\item[iii) ] $\closure\bigl[\hom(C)\bigr]=\closure\bigl[\cone(C)\bigr]$.
\item[iv) ] $\Quotient([x_1,x_2])\subset\closure[\Quotient(C)]$ for all $x_1,x_2\in C$, 
\end{itemize}
\end{lemma}

\begin{proof}
i)$\Leftrightarrow$ ii) $\Leftrightarrow$ iii) follow immediately from 
$\Quotient^{-1}(\closure[\Quotient(C)])=\closure[\hom(C)]\setminus\{0\}$ 
and from the characterization of $\cone(C)$ as the smallest convex set $K$ such that 
$C\subseteq K$ and $\hom(K)=K$. ii)$\Rightarrow$ iv) follows from the definition 
of spherical convexity. iv) $\Rightarrow$ ii): Let $y_1,y_2\in\closure[\Quotient(C)]$ and $x_i\in\Quotient^{-1}(y_i)$ $(i=1,2)$. If $x_1\sim\pm x_2$, then 
$[y_1,y_2]=\{y_1,y_2\}\subset\closure[\Quotient(C)]$. Otherwise, $x_1$ and $x_2$ are 
linearly independent, and for all $\lambda\in[0,1]$, 
$\Quotient(\lambda y_1+(1-\lambda)y_2)=\lim_{n\rightarrow\infty}\Quotient\left(\lambda x_1^n
+(1-\lambda)x_2^n\right)\in\closure\left[\Quotient(C)\right]$ 
for some sequences $(x_i^n)_{\N}\subset C$ for which $\Quotient(x_i^n)\rightarrow y_i$, 
$(i=1,2)$. 
\end{proof}

\begin{figure}
\begin{center}
\psfrag{x1}{$x_1$}
\psfrag{x2}{$x_2$}
%\resizebox{\textwidth}{!}{\includegraphics{ex.eps}}
\scalebox{0.4}{\includegraphics{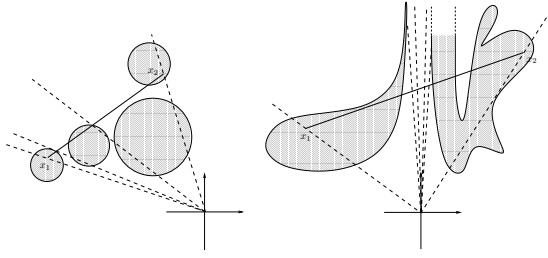}}
\end{center}
\caption{The shaded sets $C$ are homogenization-convex. For every pair of 
points $x_1,x_2\in C$, every point on the interval $[x_1,x_2]$ can be produced as a limit 
of positively scaled points in $C$. In the example on the right the two connected 
components of $C$ go off to infinity in the same asymptotic direction.}
\label{examples}
\end{figure}

\begin{figure}
\begin{center}
\psfrag{x1}{$x_1$}
\psfrag{x2}{$x_2$}
%\resizebox{\textwidth}{!}{\includegraphics{cex.eps}}
\scalebox{0.4}{\includegraphics{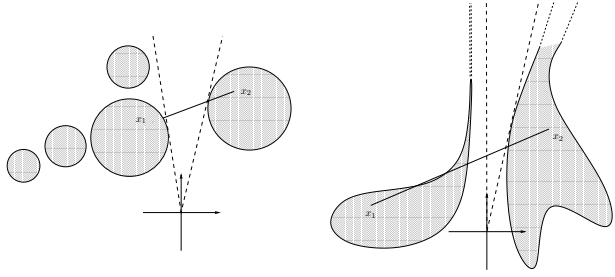}}
\end{center}
\caption{The shaded sets $C$ are not homogenization-convex. The directions 
going through the part of $[x_1,x_2]$ cut out by the bold-face dotted 
lines is not covered by any point in $\closure[\Quotient(C)]$. In the example on the right the 
two connected components of $C$ go off to infinity in different asymptotic directions.}
\label{counterexamples}
\end{figure}

It follows from Lemma \ref{lem:trivial2D} iii) that if $C$ is convex then $C$ is homogenization-convex. 
The examples of Figure \ref{examples} illustrate that the reverse relationship is not true. 
See also Figure \ref{counterexamples} for examples of sets that are {\em not} 
homogenization-convex. The following example is relevant to our proof of the S-Lemma:

\begin{example}\label{parabola}
Let $x(t)=a_0+t a_1+t^2 a_2$ and $y(t)=b_0+t b_1+t^2 b_2$ for some real coefficients $a_i,b_i$ 
$(i=0,1,2)$. Then  
$C:=\{[\begin{smallmatrix}x(t)&y(t)\end{smallmatrix}]^{\T}:\,t\in\R\}$ is homogenization-convex. 
\end{example}

\begin{proof}
When $(a_2,a_1)$ and $(b_2,b_1)$ are linearly dependent, then there exist $\eta_1,\eta_2\in\R$, not 
both zero, such that $\eta_1 a_i +\eta_2 b_i=0$ $(i=1,2)$, and then $C$ is a subset of 
the line $\{z\in\R^2:\, \eta_1 z_1+\eta_2 z_2=\eta_1 a_0+\eta_2 b_0\}$. Since $C$ is connected by 
arcs, it must be an interval, hence convex. This implies that $C$ is homogenization-convex. 
In the case where $(a_2,a_1)$ and $(b_2,b_1)$ are linearly independent, there exist $\xi_1,\xi_2\in\R$ 
such that $\xi_1(a_2,a_1)+\xi_2(b_2,b_1)=(0,1)$, so that $\xi_1 x(t)+
\xi_2 y(t)=t+c$ for some $c\in\R$. Furthermore, we may assume without loss of generality that $b_2\neq 0$. The set of loci $\{(x,y):\,x=x(t),\,y=y(t),\,t\in\R\}$ is then characterised by the equation 
\begin{equation*}
y=b_2\left(\xi_1 x+\xi_2 y-c\right)^2
+b_1\left(\xi_1 x+\xi_2 y-c\right)+b_0.
\end{equation*}
This is the general equation of a parabola. Hence, $C=\partial K$, where $K$ is the set of points 
enclosed by the parabola. $K$ being a convex set with unique recession direction $(a_2,b_2)$, 
the homogenization-convexity of $C$ is a special case of Example \ref{lem:New} below.
\end{proof}

\begin{example}\label{lem:Newnew}
Let $C=\partial K$ where $K$ is a closed convex subset of $\R^2$ with complement 
$K^c=\R^2\setminus K$ and such that $0\in\interior[K^c]$. Then $C$ is homogenization-convex. 
\end{example}

\begin{proof}
Consider the map 
\begin{align*}
\sigma:\,\R^2&\rightarrow\R\cup\{+\infty\}\\
v&\mapsto\inf\left\{\tau\geq 0:\,\tau v\in K\right\}, 
\end{align*}
defined for all $v\in K$, where $\inf\emptyset:=+\infty$ as usual. Choose 
arbitrary points $x_1,x_2\in C$. If $x_1,x_2$ 
are linearly dependent, then $\Quotient([x_1,x_2])\subseteq\{\Quotient(x_1), \Quotient(x_2)\}
\subseteq\closure[\Quotient(C)]$. Else $x_1,x_2$ are linearly independent, and for any point $x\in[x_1,x_2]$, 
we have $x\neq 0$, so that $\Quotient(x)$ is well defined. Since $x\in K$, we have $\sigma(x)\leq 1$, 
and since $0\in\interior[K^c]$, $\sigma(x)>0$. Furthermore, $\sigma(x)x\in\partial K=C$, so that 
$x=\sigma(x)^{-1}(\sigma(x)x)\in\hom(C)$. Since $x$ was chosen arbitrarily, this shows that 
$\Quotient([x_1,x_2])\subset\closure[\Quotient(C)]$, and the claim follows from Lemma 
\ref{lem:trivial2D} iv). 
\end{proof}

\begin{example}\label{lem:New}
Let $C=\partial K$ where $K$ is a closed convex subset of $\R^2$ with at most one recession 
direction. Then $C$ is homogenization-convex. 
\end{example}

\begin{proof}
We may assume without loss of generality that 
$0\in K$, for otherwise our claim is true by virtue of Example \ref{lem:Newnew}. Consider 
the map 
\begin{equation*}
\varsigma(v):=\sup\{\tau\geq 0:\,\tau v\in K\},
\end{equation*}
defined for all $v\in K$. Then $\varsigma(\cdot)$ takes finite values on 
$K\setminus\bigl(\Quotient^{-1}(\rec(K))\cup\{0\}\bigr)$. Since $\varsigma(v)v\in\partial K=C$ 
when $\varsigma(v)$ is finite, it follows that 
\begin{equation}\label{preclosure}
\hom(K)\setminus\Quotient^{-1}(\rec(K))
\subseteq\hom(C)\subseteq\hom(K).
\end{equation}
By assumption, $\rec(K)$ is either empty or a singleton. If $\dim(K)=1$, then $C=K$. 
Otherwise, taking closures in \eqref{preclosure} reveals that 
$\closure[\hom(K)]=\closure[\hom(C)]$, and by convexity of $K$, $\cone(K)\subseteq\hom(K)$. 
Therefore, 
\begin{equation*}
\closure[\hom(C)]=\closure[\cone(K)]\supseteq\closure[\cone(C)]\supseteq
\closure[\hom(C)], 
\end{equation*}
and the claim follows from Lemma \ref{lem:trivial2D} iii). 
\end{proof}

%%%%%%%%%%%%%%%%%%%%%%%%%%%%%%%%%%%%%%%
\subsection{Dual Cone Calculus}\label{sec:farkas}
%%%%%%%%%%%%%%%%%%%%%%%%%%%%%%%%%%%%%%%

Any subset $C\subseteq\R^n$ is associated with 
a dual cone $C^*=\{y\in\R^n:\,\langle x,y\rangle\geq 0,\;\forall\,x\in C\}$. When 
$K_1,K_2$ are closed polyhedral cones, then the dual cone formula 
\begin{equation}\label{dual cone calculus formula}
(K_1\cap K_2)^*=K_1^*+K_2^*
\end{equation}
applies. In particular, this formula holds true for all closed cones $K_1,K_2\subseteq\R^2$, 
since all cones in $\R^2$ are polyhedral. The following property of dual cones is also 
well known, 
\begin{align}
%C^{**}&=\closure[\cone(C)],\label{formula1}\\ 
C^*&=\left(\closure[\cone(C)]\right)^*,\label{formula2}\\
%A\subset B&\Rightarrow B^*\subset A^*.\label{formula3}
\end{align} 
In this section we set out to generalizing 
the relation \eqref{dual cone calculus formula} to the case where $K_1$ is merely a 
homogenization-convex set and $K_2$ is a closed convex cone with nonempty interior. 

\begin{lemma}\label{lemma1}
Let $C\subseteq\R^2$ be homogenization-convex and $K\subseteq\R^2$ a closed convex 
cone with nonempty interior. Then 
\begin{equation}\label{formula4}
\closure[\cone(C\cap K)]=\closure[\cone(C)\cap K].
\end{equation} 
\end{lemma} 

\begin{proof}
We only need to prove the inclusion $\supseteq$, since the reverse relation is trivial. 
Let $x\in\cone(C)\cap K\setminus\{0\}$. Then there exist $x_1,x_2\in C$ and $\lambda_1,
\lambda_2\geq 0$ such that $x=\lambda_1 x_1+\lambda_2 x_2$. If either $\lambda_1$ or 
$\lambda_2$ is zero or if $x_1,x_2\in K$, then it is trivially true that $x\in\cone(C\cap K)$. 
Furthermore, if $x_1,x_2$ are linearly dependent, then $x=\tau x_i$ for some $\tau>0$ and 
$i\in\{1,2\}$, and by homogeneity of $K$, $x_i\in C\cap K$ and $x\in\cone(C\cap K)$. We 
may therefore assume that $x_1,x_2$ are linearly independent, $\lambda_1,\lambda_2>0$, 
and that $x_1\notin K$. 

Like all closed convex cones in $\R^2$, $K$ is of the form 
$K=\{x:\,\phi_1(x)\geq 0,\,\phi_2(x)\geq 0\}$ for some linear forms $\phi_i:\R^2\rightarrow\R$, 
$(i=1,2)$. We may furthermore assume that both are nonzero, as the case $\phi_1=0=\phi_2$ 
is trivial, and the case $\phi_1\neq 0=\phi_2$ follows from a simplification of the argument 
we are about to give. Without loss of generality, we may assume that $\phi_1(x_1)<0$. Since 
$0\leq\phi_1(x)=\lambda_1\phi_1(x_1)+\lambda_2\phi_1(x_2)$, we then have $\phi_1(x_2)>0$. 

We first treat the case $\phi_2(x_1)\geq 0$. The linear independence of $x_1$ and $x_2$ 
implies that $y_1:=\xi x_1+(1-\xi)x_2\neq 0$, where $\xi=\phi_1(x_2)/(\phi_1(x_2)-\phi_1(x_1))
\in(0,1)$. By construction, $\phi_1(y_1)=0$. The homogenization-convexity of $C$ furhter implies 
$\Quotient(y_1)\in[\Quotient(x_1),\Quotient(x_2)]\subseteq\closure[\Quotient(C)]$. This shows 
the existence of a sequence $(y^n_1)_{n\in\N}\subset C$ such that $\phi_1(y^n_1)>0$ and $\Quotient(y^n_1)\rightarrow\Quotient(y_1)$. Defining $\rho:=\lambda_1/(\lambda_1+\lambda_2)$ 
and $z:=\rho x_1+(1-\rho)x_2$, we have $x=(\lambda_1+\lambda_2)z$ and 
$\phi_1(z)=(\lambda_1+\lambda_2)^{-1}\phi_1(x)\geq 0$. Since $\phi_1(y_1)=0$, it must be the case 
that $\rho\leq\xi$, so that $\eta:=\rho/\xi\in(0,1]$, and furthermore, $z=\eta y_1+(1-\eta)x_2$. 
Since $\phi_2(x_1),\phi_2(z)\geq 0$ and $y_1\in[x_1,z]$, we also have $\phi_2(y_1)\geq 0$, 
so that $y_1\in K$. Since 
$K$ has nonempty interior and $y_1\neq 0$, we have $y^n_1\in C\cap K$ for all $n\gg 1$, and 
without loss of generality, we may assume that this holds for all $n\in\N$. Next, if 
$\phi_2(x_2)\geq 0$, set $y_2=x_2$ and $y^n_2=x_2$ for all $n\in\N$. Otherwise, 
interchanging the roles of $x_1$ 
and $x_2$ and of $\phi_1$ and $\phi_2$, a repeat of the above construction yields the existence 
of a point $y_2\in K\setminus\{0\}$ and of a sequence $(y^n_2)_{n\in\N}\subset C\cap K$ such that $\Quotient(y^n_2)\rightarrow\Quotient(y_2)$ and $z\in[y_1,y_2]$. This shows 
\begin{align*}
x=(\lambda_1+\lambda_2)z&\in\cone(\{y_1,y_2\})\\
&\subseteq\closure[\cone(\{y^n_i: n\in\N, i=1,2\})]\\
&\subseteq\closure[\cone(C\cap K)].
\end{align*}

It remains to treat the case $\phi_2(x_1)<0$. In this case, $x\in K$ implies $x_2\in K$. The 
above construction can then be repeated using the point $x_1$ for both $\phi_1$ and $\phi_2$, 
revealing the existence of points $y_i\neq 0$ such that $\phi_i(y_i)=0$  and $z\in[y_i,x_2]$, 
$(i=1,2)$. Without loss of generality, we may assume that $y_2\in[y_1,x_2]$, whence $y_2\in K$ 
and there exists a sequence $(y^n_2)_{n\in\N}\subseteq C\cap K$ such that $\Quotient(y^n_2)\rightarrow\Quotient(y_2)$. We therefore have 
\begin{align*}
x=(\lambda_1+\lambda_2)z&\in\cone(\{y_2,x_2\})\\
&\subseteq\closure[\cone(\{y^n_2: n\in\N\}\cup\{x_2\})]\\
&\subseteq\closure[\cone(C\cap K)].
\end{align*}

In summary, we have established that 
$\closure[\cone(C\cap K)]\supseteq\cone(C)\cap K\setminus\{0\}$. 
Our claim now follows by taking closures on both sides of this inclusion. 
\end{proof}

We are now ready to state and prove the main result of this paper, for the purpose of 
which we are going to make the following regularity assumption, 
\begin{equation}\label{the condition}
\closure\bigl[\cone(C)\cap K\bigr]=
\closure\bigl[\cone(C)\bigr]\cap K.
\end{equation}

\begin{theorem}\label{main}
Let $C\subseteq\R^2$ be homogenization-convex and $K\subseteq\R^2$ a closed convex cone 
such that the regularity assumption \eqref{the condition} holds. Then 
\begin{equation*}
(C\cap K)^*=C^*+K^*.
\end{equation*} 
\end{theorem}

\begin{proof}
Using Lemma \ref{lemma1} and the classical dual cone calculus formulas, we find 
\begin{align*}
(C\cap K)^*&\stackrel{\eqref{formula2}}{=}
\left(\closure\left[\cone\left(C\cap K\right)\right]\right)^*\\
&\stackrel{\eqref{formula4}}{=}\left(\closure\left[\cone(C)\cap K\right]\right)^*\\
&\stackrel{\eqref{the condition}}{=}\left(\closure\left[\cone(C)\right]\cap K\right)^*\\
&\stackrel{\eqref{dual cone calculus formula}}{=}\left(\closure\left[\cone(C)\right]\right)^*+K^*\\
&\stackrel{\eqref{formula2}}{=}C^*+K^*. 
\end{align*}
\end{proof}

Next, let us give a sufficient criterion that is easier to check than Condition \eqref{the condition}. 

\begin{lemma}\label{closure2D}
Let $C\subseteq\R^2$ and $K\subseteq\R^2$ a convex cone. 
If $C\cap\interior[K]\neq\emptyset$, then Condition \eqref{the condition} holds. 
\end{lemma}

\begin{proof}
The proof works in arbitrary normed vector spaces $V$. 
We only need to prove that the inclusion $\supseteq$ holds in \eqref{the condition}, the reverse 
relation being trivial. Let $x_0\in C\cap\interior[K]$, and let $(x_n)_{\N}\subset\cone(C)$ be 
a sequence such that $x_n\rightarrow x\in K$. Then for every 
$\varepsilon>0$ we have $x_n+\varepsilon x_0\in\cone(C)\cap K$ for all $n$ large enough. 
Therefore, $x+\varepsilon x_0\in\closure[\cone(C)\cap K]$. This being true for all $\varepsilon>0$, 
we have $x\in\closure[\cone(C)\cap K]$, as claimed. 
\end{proof}

\begin{corollary}\label{lemma22D}
$C\subseteq\R^2$ be homogenization-convex, and let $\psi,\phi:\R^2\rightarrow\R$ be linear 
forms, with $\phi$ chosen such that there exists $x_0\in C$ where $\phi(x_0)>0$. Then the following 
conditions are equivalent, 
\begin{itemize}
\item[i)\;] $\psi(x)\geq 0$ for all $x\in C$ such that $\phi(x)\geq 0$, 
\item[ii)\;] there exists $\xi\geq 0$ such that $\psi(x)-\xi\phi(x)\geq 0$ 
for all $x\in C$. 
\end{itemize}
\end{corollary}

\begin{proof}
This is a special case of Theorem \ref{main} with $K=\{x:\,\phi(x)\geq 0\}$ and where the 
sufficient criterion of Lemma \ref{closure2D} applies. 
\end{proof}

Next, we lift Corollary \ref{lemma22D} into arbitrary real Hilbert spaces, resulting in the 
following result.

\begin{theorem}\label{generalized S-Lemma2D}
Let $(V,\langle\cdot,\cdot\rangle)$ be a real Hilbert space, 
$\psi,\phi:V\rightarrow\R$ continuous linear forms, 
$W:=(\ker(\phi)\cap\ker(\psi))^{\perp}$ and $\pi_W$ the orthogonal 
projection of $V$ onto $W$ along $\ker(\phi)\cap\ker(\psi)$. Let 
$C$ be a subset of $V$ such that $\phi(x_0)>0$ for some $x_0
\in C$ and such that $\pi_W C$ is homogenization-convex in $W$. 
Then the following conditions are equivalent:
\begin{itemize}
\item[i)\;] $\psi(x)\geq 0$ for all $x\in C$ such that $\phi(x)\geq 0$, 
\item[ii)\;] there exists $\xi\geq 0$ such that $\psi(x)-\xi\phi(x)\geq 0$ for all 
$x\in C$. 
\end{itemize}
\end{theorem}

\begin{proof}
Applying Corollary \ref{lemma22D} to $\phi|_{W}, \psi|_{W}$ and $\pi_W C$ on the two-dimensional 
subspace $W$, we find that i) $\Leftrightarrow$ 
$\psi|_W(x)\geq 0$ for all $x\in\pi_W C$ such that 
$\phi|_W(x)\geq 0$ $\Leftrightarrow$  $\exists$ $\xi\geq 0$ such that 
$\psi|_W(x)-\xi\phi|_W(x)\geq 0$ for all $x\in\pi_W C$ $\Leftrightarrow$ ii). 
\end{proof}

It is important to understand 
that Theorem \ref{generalized S-Lemma2D} is more than just a generalization of 
Corollary \ref{lemma22D} to arbitrary real Hilbert spaces, for rather than assuming that 
$C$ be homogenization-convex in $V$ (if the definition is appropriately extended to arbitrary 
Hilbert spaces), 
the theorem merely gets away with the weaker assumption that the projected set $\pi_W C$ 
be homogenization-convex. This distinction is crucial, as in our proof of the S-Lemma, $C$ is 
not homogenization-convex, while $\pi_W C$ {\em is} homogenization-convex due to the two 
dimensional nature of $W$. In fact, $\pi_W C$ is in general not homogenization-convex when  
$\dim(W)\geq 3$, and this is the main reason why the S-Lemma does not hold for quadratic 
functions copositive with more than one quadratic form. 

Note further that if the set $C$ is actually convex (rather than just 
homogenization-convex), Theorem \ref{generalized S-Lemma2D} becomes a special case of 
Farkas' Theorem, see \cite{rockafellar}. 

%%%%%%%%%%%%%%%%%%%%%%%%%%%%%%%%%%%%%%%
\subsection{Proof of the S-Lemma}\label{sec:S-Lemma}
%%%%%%%%%%%%%%%%%%%%%%%%%%%%%%%%%%%%%%%

Next, we shall see that, despite its Farkas flavour, Theorem \ref{generalized S-Lemma2D} is 
in fact a generalisation of the S-Lemma, and \eqref{the condition} is a weaking of the 
standard regularity assumption: denoting the set of real symmetric $n\times n$ matrices  
by $\Sym{n}$, and combining the tools developed above, we obtain a proof of 
Theorem \ref{S-Lemma}:

\begin{proof}
Let $g$ be given by $g(x)=x^{\T}Qx+2\ell^{\T}x+c$, where 
$Q\in \Sym{n}$, $\ell\in\R^n$ and $c\in\R$. Then 
$g(x)=\langle A,[\begin{smallmatrix}x&1\end{smallmatrix}]^{\T}
[\begin{smallmatrix}x&1\end{smallmatrix}]\rangle$, 
where $\langle A,X\rangle=\tr(A^{\T}X)$ is the trace inner product 
defined on the space $\Sym{n+1}$ of symmetric $(n+1)\times
(n+1)$ matrices, and where 
\begin{equation*}
A=\begin{bmatrix}Q&\ell\\ \ell^{\T}&c\end{bmatrix}.
\end{equation*} 
Likewise, there exists $B\in \Sym{n+1}$ such that 
$h(x)=\langle B,[\begin{smallmatrix}x&1\end{smallmatrix}]^{\T}
[\begin{smallmatrix}x&1\end{smallmatrix}]\rangle$. 
Let $C\subset \Sym{n+1}$ be defined by 
$C=\{zz^{\T}:\,z=[\begin{smallmatrix}x&1\end{smallmatrix}]^{\T}, 
x\in\R^n\}$. Using the notation just introduced, the claim of the theorem 
is that the following two conditions are equivalent,
\begin{itemize}
\item[i)\;] $\langle A,X\rangle\geq 0$ for all $X\in C$ such that 
$\langle B,X\rangle\geq 0$,
\item[ii)\;] there exists $\xi\geq 0$ such that $\langle A-\xi B,X\rangle\geq 0$ 
for all $X\in C$. 
\end{itemize}
We note that $\psi:X\mapsto\langle A,X\rangle$ and $\phi:X\mapsto\langle B,X\rangle$ 
are linear forms on $\Sym{n+1}$. Furthermore, if 
$X_0=[\begin{smallmatrix}x_0&1\end{smallmatrix}]^{\T}
[\begin{smallmatrix}x_0&1\end{smallmatrix}]$, then 
$X_0\in C$ and $\phi(X_0)=h(x_0)>0$. Thus, the equivalence of 
i) and ii) follows from Theorem \ref{generalized S-Lemma2D} if it can be established 
that $\pi_W C$ is homogenization-convex, where $\pi_W$ 
is the orthogonal projection of $(\Sym{n+1},\langle\cdot,\cdot\rangle)$ 
onto $W:=(\ker(\phi)\cap\ker(\psi))^{\perp}=\Span\{A,B\}$. Let $X_1,X_2\in C$. Then $X_i=[\begin{smallmatrix}x_i&1\end{smallmatrix}]^{\T}
[\begin{smallmatrix}x_i&1\end{smallmatrix}]$ 
for some $x_i\in\R^n$, $(i=1,2)$. For $t\in\R$, define $x(t):=x_2-t(x_2-x_1)$ and 
$X(t):=[\begin{smallmatrix}x(t)&1\end{smallmatrix}]^{\T}
[\begin{smallmatrix}x(t)&1\end{smallmatrix}]
=G_0+t G_1 + t^2 G_2$, where
\begin{align*}
G_0&=\left[\begin{smallmatrix}x_2\\1\end{smallmatrix}\right]
\left[\begin{smallmatrix}x_2\\1\end{smallmatrix}\right]^{\T},\\
G_1&=-\left[\begin{smallmatrix}x_2\\1\end{smallmatrix}\right]
\left[\begin{smallmatrix}x_2-x_1\\0\end{smallmatrix}\right]^{\T}
-\left[\begin{smallmatrix}x_2-x_1\\0\end{smallmatrix}\right]
\left[\begin{smallmatrix}x_2\\1\end{smallmatrix}\right]^{\T},\\
G_2&=\left[\begin{smallmatrix}x_2-x_1\\0\end{smallmatrix}\right]
\left[\begin{smallmatrix}x_2-x_1\\0\end{smallmatrix}\right]^{\T}. 
\end{align*}
Let $E_1,E_2\in \Sym{n+1}$ be an orthonormal basis 
of $W$. Then $\pi_W X(t)=a(t)E_1+b(t)E_2$, where 
$a(t)=\langle G_0,E_1\rangle +t\langle G_1,E_1\rangle 
+t^2\langle G_2,E_1\rangle$ and 
$b(t)=\langle G_0,E_2\rangle +t\langle G_1,E_2\rangle 
+t^2\langle G_2,E_2\rangle$. Definining 
$T:=\{[\begin{smallmatrix}a(t)&b(t)\end{smallmatrix}]^{\T}:\,t\in\R\}$, 
Lemma \ref{parabola} shows that $T$ is homogenization-convex in $\R^2$. By virtue of  
Lemma \ref{lem:trivial2D}\, iv), this implies that $\pi_W C$ is homogenization-convex, as 
claimed. 
\end{proof}

%%%%%%%%%%%%%%%%%%%%%%%%%%%%%%%
\subsection{Generalization to Hilbert Space Kernels}\label{sec:Hilbert}
%%%%%%%%%%%%%%%%%%%%%%%%%%%%%%%

The proof given above generalizes to infinite-dimensional spaces:

\begin{theorem}\label{S-Lemma, infinite-dimensional}
Let $(V,\langle\cdot,\cdot\rangle)$ be a real Hilbert space, and let 
$g,h:\,V\rightarrow\R$ be continuous quadratic functions defined on 
$V$ by 
\begin{align*}
g:\,&x\mapsto c_g+2\langle v_g,x\rangle+\langle x,M_g x\rangle,\\
h:\,&x\mapsto c_h+2\langle v_h,x\rangle+\langle x,M_h x\rangle,
\end{align*}
where $M_g,M_h:\,V\rightarrow V$ are self-adjoint operators, 
$v_g,v_h\in V$ and $c_g,c_h\in\R$. Let $h$ us further assume that there exists 
$x_0\in V$ where $h(x_0)>0$. Then $g$ is 
copositive with $h$ if and only if there exists $\xi\geq 0$ such that 
$g(x)-\xi h(x)\geq 0$ for all $x\in V$.
\end{theorem}

\begin{proof}
The proof is identical to that of Theorem \ref{S-Lemma} bar the 
following construction: let $H:=V\oplus\R$, where $\oplus$ denotes 
the direct sum of Hilbert spaces, and let us write  $\langle\cdot,\cdot\rangle_H$ 
for the inner product on $H$. Let $\sym$ be the space of 
self-adjoint operators on $H$. By the Hellinger-Toeplitz Theorem, such operators 
are automatically continuous, and it is easy to see that $A,B\in\sym$, where 
\begin{align*}
A:\,&(x,\tau)\mapsto(M_g x+\tau v_g,\langle v_g,x\rangle+\tau c_g)\\
B:\,&(x,\tau)\mapsto(M_h x+\tau v_h,\langle v_h,x\rangle+\tau c_h).
\end{align*}
Let $\{e_i:\,i\in\N\}$ be an orthonormal basis of $H$. The following operators are in 
$\sym$, 
\begin{equation*}
E_{ij}:\,y\mapsto\frac{1}{1+\delta_{ij}}\left(
\langle e_i,y\rangle_H e_j+\langle e_j,y\rangle_H e_i\right),
\end{equation*}
where $\delta_{ij}$ is the Kronecker delta. Defining 
\begin{equation*}
\langle E_{ij},E_{kl}\rangle_S:=\begin{cases}
1\quad\text{if }\{i,j\}=\{k,l\},\\
0\quad\text{otherwise}, 
\end{cases}
\end{equation*}
the $E_{ij}$ generate a Hilbert space $(S,\langle\cdot,\cdot\rangle_S)$ for 
which $\{E_{ij}:\,i,j\in\N\}$ is an orthonormal basis. In fact, $S$ is the set of compact 
operators in $\sym$, and the topology defined by the trace inner product $\langle\cdot,\cdot\rangle_S$ 
is the uniform topology, since $\langle E_{ij},X\rangle_S=\langle e_i, X e_j\rangle_H$ for all $X\in S$. 
Every $x\in V$ defines an operator $R(x)\in\sym$, 
\begin{equation*}
R(x):\,z\mapsto\langle(x,1),z\rangle_H(x,1),
\end{equation*}
and if $(x,1)=\sum_{i\in\N}\xi_i e_i$ then 
$R(x)=\sum_{ij}\xi_i\xi_j E_{ij}$ and 
$\sum_{ij}\xi_i^2\xi_j^2=(\sum_i\xi_i^2)
(\sum_j\xi_j^2)<\infty$. 
This shows that $C:=\{R(x):\,x\in V\}\subset S$. 
Extending the map 
\begin{align*}
\psi:\,C&\rightarrow\R,\\
R(x)&\mapsto\langle (x,1),A(x,1)\rangle_H
\end{align*}
by linearity and continuity, we obtain a bounded linear 
operator on the Hilbert space $(\closure[\Span(C)],
\langle\cdot,\cdot\rangle_S)$. Likewise, $B$ defines 
a bounded linear operator $\phi$ on the same space. 
Replacing $\sym_{n+1}$ by $\closure[\Span(C)]$ in the proof 
of Section \ref{sec:S-Lemma}, a repetition of the arguments presented there 
proves the claim of Theorem \ref{S-Lemma, infinite-dimensional}.
\end{proof}

We remark that the condition 
\begin{equation*}
g(x)-\xi h(x)\geq 0,\quad\forall\,x\in V
\end{equation*}
is equivalent to requiring that 
\begin{align*}
K:V\times V&\rightarrow\R,\\
(x,y)&\mapsto\langle x,(M_g-\xi M_h)y\rangle+\langle v_g-\xi v_h, x+y\rangle+c_g-\xi c_h
\end{align*}
be a positive definite kernel. \\

%\noindent  {\bf Acknowledgements:} 
%This research was supported through grants GR/S34472 and EP/H02686X/1 
%from the Engineering and Physical Sciences Research Council of the United Kingdom (EPSRC). 
%The author wishes to thank the referees for their constructive feedback, and he is also very grateful 
%to Paul Tseng for his professional handling of the original submission of this paper. Paul was 
%Editor of Mathematics of Operations Research before he went missing on a kayak trip in 2009. 

\end{document}